\documentclass[12pt,reqno]{amsart}

\usepackage{graphics,graphicx}
\usepackage{pstricks,pst-node,pst-tree}
\usepackage{tikz,caption,subcaption}
\usepackage{mathtools,mathrsfs}
\usepackage{amsmath, amssymb, amsfonts,amsthm}
\usepackage{textcomp,fullpage,pdfpages,newlfont,hyperref}
\usetikzlibrary {positioning}

\newcommand{\cd}{\ \stackrel{d}{\longrightarrow} \ }

\newcommand{\bone}{\mathbf 1}

\newcommand{\bw}{\mathbf{w}}

\newcommand{\FF}{\mathcal{F}}

\newtheorem{theorem}{Theorem}

\newtheorem{corollary}{Corollary}
\newtheorem{proposition}{Proposition}

\theoremstyle{theorem}
\newtheorem{definition}{Definition}

\theoremstyle{remark}
\newtheorem{remark}{Remark}

\begin{document}
\title{Negatively Reinforced Balanced Urn Schemes}
\author{Gursharn Kaur}
\address[Gursharn Kaur]{Theoretical Statistics and Mathematics Unit \\
         Indian Statistical Institute, Bangalore Centre \\ INDIA}
\email{\href{mailto:gursharn.kaur24@gmail.com}{gursharn.kaur24@gmail.com}}

\date{\today}

\selectfont
     
\begin{abstract}
We consider weighted negatively reinforced urn schemes with finitely many colours. 
An urn scheme is called negatively reinforced, if the selection probability for a colour is 
proportional to the weight $w$ of the colour proportion, where $w$ is a non-increasing function.
Under certain assumptions on the replacement matrix $R$ and weight function $w$, 
such as, $w$ is differentiable and $w(0) < \infty$, 
we obtain almost sure convergence of the random configuration of the urn model.
In particular, we show that if $R$ is doubly stochastic the random configuration of the urn converges to the uniform vector, 
and asymptotic normality holds, if the number of colours in the urn are sufficiently large.
\end{abstract}



\maketitle

\section{Introduction}
\subsection{Background and Motivation}
The classical P\'olya urn model was originally introduced by P\'olya \cite{Polya30} and since then many generalization
of the classical P\'olya urn scheme have been studied \cite{BaiHu05, Pe90, Svante1, BaTh14a, BaTh14b, maulik1, maulik2, DasMau11}. 
One such generalization where the selection of a colour at every step is done according to an increasing weight function, with a 
random replacement rule was studied by Laurelle and Pages \cite{LaPa13}.
For such non linear urn models they obtained results on the almost sure convergence and central limit theorem of the random configuration.
A different class of urn models, namely, \emph{linear negatively reinforced urn models} was introduced in \cite{BKaur17-1}, 
where selection is done according to a weight function which is linear but non-increasing. In this paper, we 
investigate a generalization of these later class of models, for general \emph{non-increasing} weight functions.
The main tool used in this paper is \emph{stochastic approximation method,} which is a powerful tool to study recursive algorithms.
Recently, Zhang \cite{Zhang2016} has provided asymptotic normality for a stochastic approximation algorithm under certain assumptions, which we use in this work.

As mentioned in \cite{BKaur17-1}, resource constrain modelling problems is one of the main motivation to study such models. 
In particular, multi-server queuing systems with capacity constrains \cite{LucMcDiar2005, LucMcDiar2006} are good examples.  
For such models a desirable outcome is a balancing of the \emph{loads}. 
In other words, it is desirable to obtain uniform load distribution at the limit. 
We will see later, that limiting uniform distribution can only be achieved, if we choose a doubly stochastic replacement matrix. 
As a result we mainly focus on the doubly stochastic matrices, and show that the almost sure convergence to uniform distribution holds under 
fairly general assumptions on the weight function. We further establish the corresponding central limit theorems.
%

\subsection{Model}
In this work, we will only consider \emph{balanced} urn schemes with $k$-colours, index by $S := \left\{1,\ldots, k\,\right\}$. 
We essentially work under the framework introduced in \cite{BKaur17-1}. 
For the sake of completeness, we provide here the complete description of the model, which is exactly similar to what is 
presented in \cite{BKaur17-1}, except we use more general non-increasing weight functions. 

We denote by $R := \left(\left(R_{i,j}\right)\right)_{1\leq i,j \leq k}$ the \emph{replacement matrix}, 
that is, 
$R_{i,j} \geq 0$ is the \emph{number of balls of colour $j$ to be placed in the the urn when the colour of the selected ball 
is $i$}. The model will be called \emph{balanced}, if all the row sums of $R$ are constant. In that case, 
dividing the entries of $R$ by the common row total, without loss we may assume $R$ is a \emph{stochastic matrix}. 
We will also assume that the starting configuration $U_0:=\left(U_{0,j}\right)_{1 \leq j \leq k}$ 
is a probability distribution on the set of colours $S$. As we will see from the proofs of our main results, 
this apparent loss of generality can easily be removed. For simplicity, in this work we also assume $U_{0,j} > 0$ for every
$1 \leq j \leq k$. 

Denote by $U_n := \left(U_{n,j}\right)_{1 \leq j \leq k} \in [0, \infty)^k$
the random configuration of the urn at time $n$. Also let $\FF_n := \sigma\left(U_{0}, U_{1}, \cdots, U_{n}\right)$ be the
natural filtration. The $(n+1)$-th randomly selected colour will be denoted by $Z_n$, which has the conditional distribution
given $\FF_n$ as 
\begin{equation}
P\left( Z_n = j \,\Big\vert\, \FF_n \right) 
\propto 
w \left(\frac{U_{n,j}}{\sum_{i=1}^k U_{n,i}}\right), \,\,\, 1 \leq j \leq k.
\label{prob-weight}
\end{equation}
where $w:[0,1] \to \mathbb{R}^+$ is a \emph{non-increasing} function. 

Starting with $U_0$ we define $\left(U_n\right)_{n \geq 0}$ recursively as follows:
\begin{equation}
U_{n+1} = U_n + \chi_{n+1} R.
\label{Equ:Basic-Recursion} 
\end{equation}
where $\chi_{n+1} := \left(\bone\left(Z_n = j \right)\right)_{1 \leq j \leq k}$. 

We call the process $\left(U_n\right)_{n \geq 0}$, a \emph{negatively reinforced urn scheme} with
initial configuration $U_0$ and replacement matrix $R$. 
In this work, we will be interested in studying the asymptotic properties of the following two processes: \\

\noindent
{\bf Random configuration of the urn:}
Observe that for all $n \geq 0$,
\begin{equation}
\sum_{j=1}^{k} U_{n,j} = n + 1.
\label{Equ:Sum-of-weights}
\end{equation}
This holds because $R$ is a stochastic matrix and $U_0$ is a probability vector. 
Thus the \emph{random configuration of the urn}, namely, $\displaystyle{\frac{U_n}{n+1}}$
is a probability mass function. \\

\noindent
{\bf Colour count statistics:}
Let $N_n := \left(N_{n,1}, \hdots , N_{n,k}\right)$ be the vector of length $k$,
whose $j$-th element is the number of times colour $j$ was selected in the first $n$ trials, that is
\begin{equation}
\label{Def:Nn}
N_{n,j} = \sum_{m=0}^{n-1} \bone\left(Z_m = j\right), \,\,\, 1 \leq j \leq k.
\end{equation}
It is easy to note that from  ~~\eqref{Equ:Basic-Recursion} it follows
\begin{equation}
\label{Un-Nn}
U_{n+1} = U_0 + N_{n+1} R. 
\end{equation}

\subsection{Outline}
In Section \ref{SAalgorithm} we first establish a relation between the urn model and the stochastic approximation
algorithm, then in Section \ref{Main-results-AS} and Section \ref {Main-results-CLT} we present our main results.
In Section \ref{Sec:Technical-Results} we give the technical results and proofs of the main results are given in Section \ref{Section-Proofs}.
In Section \ref{Section-Examples} some examples are given.

\section{Stochastic Approximation and Urn}\label{SAalgorithm}
In this section we first define a stochastic approximation algorithm.
A stochastic approximation algorithm $\left( X_n \right)_{n\geq 0}$ is defined as a stochastic process in $\mathbb{R}^k$, given by 

\begin{equation}\label{SA-recursion}
X_{n+1} =X_n + \gamma_{n+1} h\left(X_n\right) + \gamma_{n+1} M_{n+1}, \;\;\;\; n \geq 0 
\end{equation} 
where $h: \mathbb{R}^k \to \mathbb{R}^k$ and 
\begin{enumerate}
\item[(i)] $\left(\gamma_n\right)_{n\geq 1}$ is a sequence of positive real numbers, such that,
\begin{equation}\label{Step-size}
\sum_{n=1}^\infty \gamma_n = \infty \;\;\text{ and } \;\; \sum_{n=1}^\infty \gamma_n^2 < \infty . 
\end{equation}

\item[(ii)] $\left( M_n\right)_{n\geq 1}$ is a square integrable martingale difference sequence with respect to the filtration 
$\mathcal{G}_n = \sigma\{ X_m, M_m, \; m\leq n\}$ and there exists a constant $ C > 0$, such that
\begin{equation}\label{Martingale-bound}
E\left[\|M_{n+1}\|^2 |\mathcal{G}_n \right] \leq C\left(1+\|X_n\|^2\right) \;a.s. 
\end{equation}
for $n\geq 0$.
\end{enumerate}
In the next two subsections we will show that the vector of colour proportions and colour count proportions 
can be written as a stochastic approximation algorithm.

\subsection{Stochastic Approximation Algorithm for the Random Urn Configuration}
We put
\begin{equation}
Y_n \coloneqq \dfrac{U_n}{n+1},  
\end{equation}
which is the vector of colour proportions at time $n$. Observe that, for the urn model defined in equation \eqref{Equ:Basic-Recursion}, we have
\begin{equation}\label{Conditional-Mean}
E\left[U_{n+1}-U_n|\mathcal{F}_n\right] = \frac{\bw (Y_n)}{S_w(Y_n)}R
\end{equation}
where $\bw(Y_n) = \left( w(Y_{n,1}), w(Y_{n,2}), \cdots, w(Y_{n,k}) \right)$, and $S_w(Y_n) =\sum_{i=1}^k  w(Y_{n,i})$.

Therefore the recurrence relation in equation \eqref{Equ:Basic-Recursion} can be written as 
\begin{align}	
U_{n+1}& = U_n + E[\chi_{n+1}|\mathcal{F}_n]R +\big[ \chi_{n+1} -E[\chi_{n+1}|\mathcal{F}_n] \big]R \nonumber \\
&= U_n + \frac{\bw(Y_n)}{S_w(Y_n)}R + M_{n+1}R\label{Recursion-martingale}
\end{align}
where $ M_{n+1} = \chi_{n+1} -E\left[\chi_{n+1}|\mathcal{F}_n\right]$ is an $\mathcal{F}_n$ martingale difference. Now observe that
\begin{align*}
\frac{U_{n+1}}{n+2}&= \frac{U_n}{n+2} + \frac{1}{n+2}\frac{\bw(Y_n)}{S_w(Y_n)}R + \frac{1}{n+2} M_{n+1}R\\
\implies Y_{n+1}&= Y_n\frac{n+1}{n+2} + \frac{1}{n+2}\frac{\bw(Y_n)}{S_w(Y_n)}R + \frac{1}{n+2} M_{n+1}R \\
\implies Y_{n+1} &= Y_n + \frac{1}{n+2}\left( \dfrac{\bw(Y_n)}{S_w(Y_n)}R -Y_n\right) + \frac{1}{n+2} M_{n+1}R 
\end{align*}
which is exactly of the form given in the equation \eqref{SA-recursion}, that is the urn configuration $Y_n$ can be written 
as a $k$-dimensional stochastic approximation algorithm given by:
\begin{equation}
Y_{n+1}= Y_n + \gamma_{n+1} h(Y_n) +\gamma_{n+1} M_{n+1}R
\label{Urn-Recursion}
\end{equation} 
where $\gamma_n = \dfrac{1}{n+1}$, and $h : \mathbb{R}^k \to \mathbb{R}^k$ is given by
\begin{equation}
h(y) = \frac{\bw(y)}{S_w(y)}R - y.
\label{Urn-function}
\end{equation}
where we extend the function $w$ continuously to whole of $\mathbb{R}$, by making it a constant function outside the interval $[0,1]$, that is,
$w(y) = w(0) $ for $ y\leq 0$ and $w(y) = w(1) $ for $ y\geq 1$. Also note that $\gamma_n \sim \mathcal{O}\left(n^{-1}\right)$ 
satisfies the required conditions given in (ii) and $\left(M_{n}R\right)_{n\geq 1}$ is a martingale difference sequence that is 
\begin{equation}
E\left[M_{n+1}R\vert\mathcal{F}_n\right] = 0, \;\;\forall n \geq 0.
\end{equation}
and since $R$ is a stochastic matrix we get
\begin{align*} 
\|M_{n+1} R\|^2 & = \sum_{i=1}^k \left|\sum_{j=1}^k M_{n+1,j}R_{j,i}\right |^2 \\
&\leq \sum_{i=1}^k \sum_{j=1}^k \left|M_{n+1,j}\right |^2 \\
&= k \sum_{j=1}^k \left|\chi_{n+1,j} - E[\chi_{n+1,j}|\mathcal{F}_n]\right |^2 \\
&\leq k \sum_{j=1}^k \left|\chi_{n+1,j}\right|^2 + E[\chi_{n+1,j}|\mathcal{F}_n]^2 
\end{align*}
Now since $\chi_{n+1,j}^2 = \chi_{n+1,j}$ (as it only takes value $0$ or $1$) and $\sum_{j=1}^k\chi_{n+1,j} = 1$, therefore we get
\begin{align}
E\left[\|M_{n+1} R\|^2 |\mathcal{F}_n \right] &\leq k \left(1+ \sum_{j=1}^k E[\chi_{n+1,j}|\mathcal{F}_n]^2 \right) \nonumber \leq k\left(1+ k \right). \label{Bound-h-Martingale}
\end{align}
Thus $\left(M_{n}R\right)_{n\geq 1}$ also satisfies the conditions given in equation \eqref{Martingale-bound}.
Therefore, the ODE associated to \eqref{Urn-Recursion} is 
\begin{equation}\label{ODE-Yn}
\dot{y} = h(y)
\end{equation}
where $h$ is given in equation \eqref{Urn-function}.

\subsection{Stochastic Approximation Algorithm for the Colour Count Statistics}
\label{Sec:colour-Counts}
Recall that from the definition $\sum_{i=1}^k N_{n,i} = n$, so we denote the colour count proportions by 
\[
\displaystyle{\tilde{Y}_n\coloneqq \frac{N_n}{n}}
\]
Note that we can write
\begin{align}
N_{n+1} &= N_n +\chi_{n+1}\nonumber \\
&= N_n +E\left[\chi_{n+1}|\mathcal{F}_n\right]+ \left(\chi_{n+1} -E\left[\chi_{n+1}|\mathcal{F}_n\right]\right)\nonumber \\
&= N_n + \frac{\bw(Y_n)}{S_w(Y_n)} + M_{n+1} \nonumber \\
\frac{N_{n+1}}{n+1} &= \frac{N_n}{n} + \frac{1}{n+1}\left[\frac{\bw(Y_n)}{S_w(Y_n)}-\frac{N_n}{n} \right] + \frac{1}{n+1} M_{n+1}\nonumber\\
\implies\;\;\;\;\;\; \tilde{Y}_{n+1}&= \tilde{Y}_n + \frac{1}{n+1}\left[\frac{\bw\left(Y_n\right)}{S_w\left(Y_n\right)} -\tilde{Y}_n \right] + \frac{1}{n+1} M_{n+1} \label{SA-Nn-Yn}
\end{align}
Using \eqref{Un-Nn} we get
\begin{equation}
Y_n = \frac{1}{n+1}Y_0 + \frac{n}{n+1}\tilde{Y}_n R = :\tilde{Y}_nR +\delta_n
\label{Equ:Un-Nn} 
\end{equation}
for 
\begin{equation*}
\delta_n = \frac{1}{n+1}Y_0 - \frac{1}{n+1}\tilde{Y}_n R 
\end{equation*}
Therefore we can rewrite equation ~\eqref{SA-Nn-Yn} as 
\begin{align} 
\tilde{Y}_{n+1}&= \tilde{Y}_n +\frac{1}{n+1}\left[\frac{\bw\left(\tilde{Y}_nR +\delta_n \right)}{S_w\left(\tilde{Y}_nR +\delta_n\right)} - \tilde{Y}_n\right] + \frac{1}{n+1} M_{n+1} \nonumber \\
&= \tilde{Y}_n + \frac{1}{n+1}\left[\frac{\bw\left(\tilde{Y}_nR \right)}{S_w\left(\tilde{Y}_nR\right)}- \tilde{Y}_n\right] + \frac{1}{n+1}\epsilon_n+ \frac{1}{n+1} M_{n+1}
\label{SA-Nn} 
\end{align}
where 
\[ \epsilon_n = \frac{\bw\left(\tilde{Y}_nR +\delta_n \right)}{S_w\left(\tilde{Y}_nR +\delta_n\right)} -\frac{\bw\left(\tilde{Y}_nR \right)}{S_w\left(\tilde{Y}_nR\right)}.\]
Therefore $\tilde{Y}_n$ can also be written as a stochastic approximation algorithm as given in equation \eqref{SA-recursion}.
Since $\delta_n \to 0$, $\epsilon_n \to 0$, as $n \to \infty$, the ODE associated to ~\eqref{SA-Nn} is 
\begin{equation}\label{ODE-Nn}
\dot{\tilde{y}} = \tilde{h}\left(\tilde{y}\right)
\end{equation}
where $\tilde{h}: \mathbb{R}^k \to \mathbb{R}^k$ is such that
\begin{equation} \label{Count-function}
\tilde{h}\left(\tilde{y}\right) = \frac{\bw\left(\tilde{y}R \right)}{S_w\left(\bw(\tilde{y}R)\right)}-\tilde{y}. 
\end{equation}
For stochastic approximation technique, we will mostly refer to the work of Bena\"im \cite{Benaim99}, 
Kushner-clark \cite{KushnerClark78} and Borkar \cite{Borkar2008}.
In the next two sections we state our main results for the non-linear weight functions.

\section{Almost sure convergence}\label{Main-results-AS} 
In this section, the almost sure convergence of the random processes $\left(Y_n\right)_{n\geq 0}$ and $\left(\tilde{Y}_n\right)_{n\geq 0}$
are obtained under different sufficient conditions. Before stating our main results we need the following two definitions:
\begin{definition}
A function $f:\mathbb{R}^k\to \mathbb{R}^k $ is called Lipschitz, if there exists a finite real number $C'$ such that 
\begin{equation}
\|f(x)-f(y)\|\leq C' \|x-y\| ;\;\;\forall x,y \in \mathbb{R}^k.
\end{equation}
For a Lipschitz function $f$, the Lipschitz constant is defined as
\begin{equation}
M\coloneqq \sup_{x\neq y}\frac{\|f(x) - f(y)\|}{\|x-y\|}
\end{equation}
and such functions will be referred as $Lip(M)$. The function $f$ is called a \emph{contraction} if $M <1$.
\end{definition}

\begin{definition}
An equilibrium point $y^*$ of the differential equation $\dot{y}(t) = h\left(y(t)\right)$ is a
point for which $h(y^*) = 0. $
\end{definition}
Note that, for the $h$ function given in equation \eqref{Urn-function}, $y^*$ is an equilibrium point if 
\begin{equation}
h(y^*)=0\iff \bw(y^*)R =S_w(y^*)y^*. 
\end{equation}

To start with, we need the ODE in equation \eqref{ODE-Yn} and \eqref{ODE-Nn} to have a unique solution. 
A sufficient condition for the ODEs to have a unique solution, is when $h$ and $\tilde{h}$ are Lipschitz functions.
We will assume throughout this paper that the function $w$ is continuously differentiable, 
which implies that the function $h$ and $\tilde{h}$ are both Lipschitz and this
ensures that the associated ODEs have unique solution for any initial vector $Y_0$. 

The equilibrium points of $h$ are important as they are possible limit points for the solution of the ODEs.
For a nonlinear weight function $w$ the unique equilibrium point is guaranteed assuming that 
the function $F: \mathbb{R}^k \to \mathbb{R}^k$  defined as 
\begin{equation}
F(y) \coloneqq \dfrac{\bw(y)}{S_w(y)}R,
\end{equation}
is a contraction map. We now present the results depending on whether $F$ is a contraction.

\subsection{$F$ is a contraction}

\begin{theorem}\label{Thm:As-Convergence}
Suppose $w$ is a non-increasing weight function and $F$ is a contraction map then 
\begin{equation} \label{AS-Convergence-non-linear}
Y_n \longrightarrow y^* \;\;a.s., \;\;\text{and} \;\; \tilde{Y}_n \longrightarrow \tilde{y}^* \;\;\;a.s.
\end{equation}
where $y^*$ is the unique fixed point of $F$ and 
\begin{equation}
\tilde{y}^* = \frac{w(y^*)}{S_w(y^*)}. 
\end{equation}
In particular, convergence in \eqref{AS-Convergence-non-linear} holds, whenever non-increasing function
$w$ is a Lip(M) function and $\sqrt{k}> \dfrac{2M}{w(1)}$. 
\end{theorem}

\begin{corollary}
From equation ~\eqref{Equ:Un-Nn} we get 
\[y^* = \tilde{y}^*R,\]
or \[ \tilde{y}^* = \frac{w(y^*)}{S_w(y^*)}\]
where $y^*$ and $\tilde{y}^*$ are given in Theorem ~\ref{Thm:As-Convergence}.
\end{corollary}

\subsection{$F$ is not a contraction}
In the case when $F$ is not a contraction, we will only consider doubly stochastic replacement matrices. 
We start with the following observation.
\begin{proposition}
The uniform vector $\frac{1}{k}\bone$ is an equilibrium point of the ODE in equation \eqref{ODE-Yn}, if and only if, $R$ is a doubly stochastic matrix.
\end{proposition}

\begin{proof}
Note that,
\begin{align*}
h\left(\frac{1}{k}\bone \right) =0 &\iff \frac{\bw\left(\frac{1}{k}\right)}{S_w\left(\frac{1}{k}\bone\right)}R = \frac{1}{k}\bone \\ 
&\iff \frac{1}{k}\bone R = \frac{1}{k}\bone
\end{align*}
Thus, uniform is an equilibrium point, if and only if, $R$ is a doubly stochastic matrix.
\end{proof}

Assuming that $R$ is doubly stochastic, $\dfrac{1}{k}\bone$ is an equilibrium point for both the ODEs given in equation \eqref{ODE-Yn} and \eqref{ODE-Nn}, that is 
\begin{equation}
h\left(\frac{1}{k}\bone\right) = 0\;\;\;\text{and }\;\;\; \tilde{h}\left(\frac{1}{k}\bone\right) = 0
\end{equation} 
where $h$ and $\tilde{h}$ are defined in equation \eqref{Urn-function} and \eqref{Count-function}. 
In the next theorem, we show that for a doubly stochastic 
replacement matrix $R$ the random urn configuration converges almost surely. 

\begin{theorem}\label{Thm:DS-SLLN}
Let $w$ be a non-increasing weight function and $R$ be a doubly stochastic replacement matrix, such that for every eigenvalue $\lambda$ of $R$ 
\begin{equation}
\Re(\lambda)> \frac{kw\left(\frac{1}{k}\right)}{w'\left( \frac{1}{k}\right)},
\label{Cond:Stability} 
\end{equation}
where $\Re(\lambda)$ denotes the real part of the eigenvalue $\lambda$, then as $n \to \infty$
\begin{equation} \label{Equ:DS-Convg}
Y_n \longrightarrow \dfrac{1}{k} \bone \;\; a.s. \;\;\;\text{ and }\;\;\; \tilde{Y}_n \longrightarrow \dfrac{1}{k} \bone \;\; a.s.
\end{equation}
\end{theorem}

\section{Scaling Limits}\label{Main-results-CLT}
In this section, we will state the central limit theorems for $\left(Y_n\right)_{n\geq 0}$ and $\left(\tilde{Y}_n\right)_{n\geq 0}$.
Throughout this section we will consider the following two assumptions
\begin{enumerate}
\item[{\bf (A1)}] $w$ is a differentiable function.
\item[{\bf (A2)}] $Y_n$ converges almost surely to the uniform vector $\dfrac{1}{k}\bone.$
\end{enumerate}

We will again use the stochastic approximation method to obtain central limit theorems.
The rate of convergence of the discrete stochastic approximation process depends on the eigenvalues of the 
Jacobian matrix when evaluated at the limiting vector. 
For the ODE associated with $Y_n$, the Jacobian matrix of $h$ at the equilibrium point $\frac{1}{k} \bone$ is given by 

\begin{equation}\label{Equ:Jacobian}
\frac{\partial h(y)}{\partial y} = \frac{\partial}{\partial{y}} \dfrac{\bw(y)}{S_w(y)} \,R -I 
\end{equation}
where, 
\begin{align*}
\frac{\partial}{\partial{y}} \dfrac{\bw(y)}{S_w(y)} 
&=\begin{bmatrix}
\dfrac{w'(y_1)}{S_w(y)}-\dfrac{w(y_1)w'(y_1)}{S_w(y)^2} &-\dfrac{w(y_1)w'(y_2)}{S_w(y)^2} & \cdots&-\dfrac{w(y_1)w'(y_k)}{S_w(y)^2} \\ \\
-\dfrac{w(y_2)w'(y_1)}{S_w(y)^2} &\dfrac{w'(y_2)}{S_w(y)}-\dfrac{w(y_2)w'(y_2)}{S_w(y)^2} & \cdots&-\dfrac{w(y_2)w'(y_k)}{S_w(y)^2} \\ \\
\vdots	&\vdots&\ddots&\vdots\\ \\
-\dfrac{w(y_k)w'(y_1)}{S_w(y)^2} &-\dfrac{w(y_k)w'(y_2)}{S_w(y)^2} & \cdots&\dfrac{w'(y_k)}{S_w(y)}-\dfrac{w(y_k)w'(y_k)}{S_w(y)^2} 
\end{bmatrix} 
\end{align*}
That is, 
\begin{equation}
\frac{\partial \bw(y)/S_w(y)}{\partial y} = diag\left(\frac{\bw'(y)}{S_w(y)}\right)+ \left(\left(\dfrac{-w(y_i)w'(y_j)}{S_w(y)^2}\right)\right)_{i,j = 1,2,\cdots, k}. 
\end{equation}
Therefore

\begin{align}
\frac{\partial h(y)}{\partial y}\Big\vert_{y =\frac{1}{k}\bone} &= \left(bI -\frac{b}{k}J\right)R -I\nonumber \\
& = bR -\frac{b}{k}J -I \label{Jacobian-h}
\end{align}
where  $J\equiv J_k \equiv \bone ^T\bone$ and $I\equiv I_k $ is the $k\times k$ identity matrix and
\begin{equation} \label{Def:constant-b}
b \coloneqq \frac{w'\left(\frac{1}{k}\right)}{kw\left(\frac{1}{k}\right)}.
\end{equation}
Note that $b\leq 0$ as $w$ is a non-increasing function.
Now, since $R$ is a stochastic matrix, it has maximal eigenvalue $1$ and suppose the remaining $s$ distinct eigenvalues are $\lambda_1,\lambda_2,\cdots,\lambda_s$. 
By Perron Frobenius Theorem, the stochastic matrix $R$ has maximal eigenvalue $1$. 
That is the absolute real part of all eigenvalue of a stochastic matrix $R$ is less than $1$, so without loss we assume 
$1>\Re(\lambda_1)\geq \Re(\lambda_2)\geq \cdots \geq \Re(\lambda_s)\geq -1.$
Note that the right eigenvector corresponding to the maximal eigenvalue $1$ of $R$ is $\bone^T$ and 
\[ Dh\left(\dfrac{1}{k}\bone \right)\bone^T = \left(bI-\frac{b}{k}J\right)R\bone^T - \bone^T = -\bone^T\]
Thus, $-1$ is an eigenvalue of $Dh\left(\dfrac{1}{k}\bone \right)$. 
Now, for an eigenvalue $\lambda_i(\neq 1)$ of $R$, and the corresponding right eigenvector $v_i^T$ which is orthogonal to $\bone^T$, we have 
\[ Dh\left(\dfrac{1}{k}\bone \right)v_i^T = \left(bI-\frac{b}{k}J\right)R v_i^T - v_i^T = (b\lambda_i-1)v_i^T\]
Therefore the Jacobian matrix $Dh\left(\dfrac{1}{k} \bone\right) $ has eigenvalues $b \lambda_i-1$ for every $i = 1, \cdots, s$. 
Now define
\begin{equation}
\rho \coloneqq \max\{0,1-b \Re(\lambda_s)\}.
\label{Def:rho}
\end{equation}

\noindent
For the ODE associated to the colour count proportions $\tilde{Y}_n$,  
\begin{equation}
\frac{\partial \tilde{h}(\tilde{y})}{\partial \tilde{y}} = \frac{\partial }{\partial y}\frac{\bw\left(y \right)}{S_w\left(y\right)}R-I = \frac{\partial h(y) }{\partial y}. 
\end{equation}
Note that the $Dh\left(\dfrac{1}{k} \bone\right)$ is a diagonal matrix, if and only if,
\[ -b R_{i,j} +b/k = 0 \;\;\;\forall i \neq j \;\; \iff\;\;R = \dfrac{1}{k} J.\]
In fact, for this choice of $R$, we have 
\[ U_{n+1,i} =U_{0,i}+\frac{n+1}{k},\;\;\; \forall \;i =1,2,\cdots, k\]
so that 
\[ Y_{n+1,i} = \frac{U_{0,i}}{n+1}+\dfrac{1}{k}\]
Thus for any weight function $w$, we have
\[ Y_{n+1,i} \to \dfrac{1}{k}\;\;\;\text{ as } n \to \infty, \;\;\forall \;i. \]

We now use the CLT results obtained for general Jacobian matrix $Dh\left(\frac{1}{k}\bone\right)$, using stochastic approximation by Zhang \cite{Zhang2016}.
For the results stated in the next section, we recall that the exponential of a matrix $A$ is defined as, $ e^A \coloneqq \sum_{l=0}^\infty \dfrac{A^l}{l!}$
and for $x\in \mathbb{R}$ and a matrix $A$, $x^A$ is defined as $\exp\left((\log x)A \right)$.

\subsection{The case $\rho >1/2$}

\begin{theorem} \label{Thm:CLT-rho>1/2}
Suppose $w$ is a non-increasing function and $R$ is a doubly stochastic matrix 
such that $\rho >1/2$, then under assumptions {\bf (A1)} and {\bf (A2)},
\begin{equation} 
\sqrt{n} \left(Y_n - \frac{1}{k}\bone \right) \implies N(0, \Sigma_1) 
\end{equation}
and \begin{equation} 
\sqrt{n} \left(\tilde{Y}_n - \dfrac{1}{k} \bone \right) \implies N\left(0, \tilde \Sigma_1\right) 
\end{equation}
with
\begin{equation} \label{Def-Sigma}
\tilde{\Sigma}_1 = \frac{1}{k} \left[ \Lambda_1 - \frac{1}{k(1-2b)}J\right] \;\;\;\text{and }\;\;\; \Sigma_1 = R^T\tilde{\Sigma}_1R
\end{equation}
where $\Lambda_1$ is the unique solution of the Sylvester's equation (see \cite{RBhatia})
\begin{equation}\label{Equ:Sylvesters}
A\Lambda_1- \Lambda_1A^T = I
\end{equation}
for $A = \dfrac{1}{2}I -bR^T$. In particular, if $R$ is a normal matrix then 
\begin{equation}
\Sigma_1 = \frac{1}{k} \left[ R^T \left(I-b(R^T+R) \right)^{-1}R - \frac{1}{k(1-2b)}J\right]
\end{equation}
where $b$ is defined in equation \eqref{Def:constant-b}.
\end{theorem}

\begin{remark}
Note that for a P\'olya type urn, that is when $R = I$, assumption {\bf (A2)} holds and $\rho = 1-b >\frac{1}{2}$, therefore under assumption {\bf (A1)},
Theorem \ref{Thm:CLT-rho>1/2} holds with
\begin{equation}\label{Equ:Sigma-Polya}
\Sigma_1 = \frac{1}{k(1-2b)} \left [I-\frac{1}{k}J\right] =\frac{1}{1-2b} \Gamma,
\end{equation}
where $\Gamma = \dfrac{1}{k} I -\dfrac{1}{k^2} J$.
\end{remark}

\subsection{The case $\rho =1/2$}
Note that 
\[ \rho = \dfrac{1}{2} \iff\Re(\lambda_s) = \frac{kw\left(\frac{1}{k}\right)}{2w'\left(\frac{1}{k}\right)}\]
and since $\Re(\lambda_s) \geq -1 $ thus, $\rho = \dfrac{1}{2}$ case is possible only when $kw\left(\frac{1}{k}\right) \leq -2w'\left(\frac{1}{k}\right)$.\\
Let $\nu$ be the multiplicity of eigenvalue $\lambda_s$. 

\begin{theorem}\label{Thm:CLT-rho=1/2}
Let $w$ be a non-increasing, twice differentiable weight function such that $\rho =1/2$, then
under assumption {\bf (A2)}, 
\begin{equation}\label{Sigma-tilde}
\frac{\sqrt{n}}{(\log n)^{\nu-1/2}}\left( Y_n - \dfrac{1}{k}\bone \right)\implies N\left(0,\Sigma_2\right) 
\end{equation}
and 
\begin{equation}
\frac{\sqrt{n}}{(\log n)^{\nu-1/2}}\left( \tilde{Y}_n - \dfrac{1}{k}\bone \right)\implies N\left(0,\tilde{\Sigma}_2\right) 
\end{equation}
where
\begin{equation} 
\tilde \Sigma_2 = \frac{1}{k}\Lambda_2, \;\;\;\text{and} \;\;\;\; \Sigma_2 = R^T \tilde{\Sigma}_2 R, 
\end{equation}
and 
\begin{equation}
\Lambda_2 = \lim_{n \to \infty} \frac{1}{(\log n )^{2\nu-1}}\int_0^{\log n } e^{-u}e^{buR^T} e^{buR} du. 
\end{equation}
\end{theorem}

\subsection{The case $\rho <1/2$}
Note that 
\[ \rho < \dfrac{1}{2} \iff\Re(\lambda_s) < \frac{kw\left(\frac{1}{k}\right)}{2w'\left(\frac{1}{k}\right)}\]
and thus, $\rho < \dfrac{1}{2} $ case is not possible whenever $kw\left(\frac{1}{k}\right)>-2w'\left(\frac{1}{k}\right)$, which is true for sufficiently large $k$,
assuming that $w(0)$ and $w'(0)$ are both finite. Therefore for a negatively reinforced urn scheme, $\rho <\dfrac{1}{2}$ is a rare case, and
in this case we have the following convergence result.

\begin{theorem}\label{Thm:rho<1/2} Let $w$ be a non-increasing weight function
which is twice differentiable and $R$ be a doubly stochastic matrix, such that $0<\rho <1/2$,
then under assumption {\bf (A2)}, there are complex random variables $\xi_1, \cdots,\xi_s$ such that 
\begin{equation}
\frac{n^\rho}{\log n^{\nu-1}}\left( Y_n - \dfrac{1}{k}\bone \right) - X_n \longrightarrow 0 \;\;a.s.
\end{equation}
where,\[X_n = \sum_{i: \Re(\lambda_i) = (1-\rho)/b }e^{-i(1-bIm(\lambda_i)\log n }\xi_i v_i\]
and $v_i$ is a right eigenvector of $Dh\left(\dfrac{1}{k} \bone\right)$ with respect to the eigenvalue $b\lambda_i -1$. 
\end{theorem}

\section{Technical Results} \label{Sec:Technical-Results}
Since we study the convergence of $\left(Y_n\right)_{n\geq 0}$ and $\left(\tilde{Y}_n\right)_{n\geq 0}$  through the stochastic approximation method,
we consider the equilibrium points of $h$ as possible limit points. Now, as mentioned earlier, 
a unique equilibrium point is guaranteed assuming that $F$  is a contraction.
In the next Proposition, we obtain sufficient conditions under which $F$ is a contraction map.
\begin{proposition}
Suppose $w$ is a non-increasing $Lip(M)$ function, then $F$ is a contraction 
whenever 
\begin{enumerate}
\item[(i)] $w(1)>0$ and $ \sqrt{k}>\dfrac{2M}{w(1)}$; or
\item[(ii)] $w$ is a convex function and $ \sqrt{k}w(1/k)>2M$.
\end{enumerate}
\label{Prop:F-contraction-sufficient-condition}
\end{proposition}

\begin{remark}
If $w$ is a non-increasing convex weight function and $w(0)<\infty$, then 
$F$ is a contraction whenever $ \sqrt{k}>\dfrac{4M}{w(0)}$, for $k$ sufficiently large such that $ w(1/k)>w(0)/2$.
\end{remark}

\begin{proof}
\begin{align*}
\|F(x)-F(y)\| &= \left\|\frac{\bw(x)}{S_w(x)} -\frac{\bw(y)}{S_w(y)} \right\| \\
&=\left \|\frac{\bw(x)S_w(y) - \bw(y)S_w(x)}{S_w(x)S_w(y)}\right\| \\
&= \left\|\frac{\left(S_w(y)-S_w(x)\right)\bw(x)-S_w(x)\left(\bw(y)-\bw(x)\right)}{S_w(x)S_w(y)}\right\| \\
&\leq\frac{|S_w(y)-S_w(x)|\,\, \left\|\bw(x)\|\, +\,S_w(x)\,\,\|\bw(y)-\bw(x)\right\|}{S_w(x)S_w(y)}
\end{align*}
Note that
\[\|\bw(x)\|^2 = \sum_{i=1}^k \left|w(x_i)\right|^2 \leq \left(\sum_{i=1}^k w(x_i)\right)^2= S_w(x)^2\]
\begin{equation}
\implies \|\bw(x)\| \leq S_w(x).
\end{equation}
Therefore,
\begin{equation}\label{Lip-Bound}
\|F(x)-F(y)\|\leq \frac{|S_w(y)-S_w(x)|+\|\left(\bw(x)-\bw(y)\right)\|}{S_w(y)} 
\end{equation}
Now since $w$ is a $Lip(M)$ function we get
\begin{align}
\|\bw(x)-\bw(y)\|^2 &= \sum_{i=1}^k | w(x_i)-w(y_i)|^2 \nonumber\\
&\leq M^2\sum_{i=1}^k | x_i-y_i|^2 = M^2\|x-y\|^2 \label{Ineq:Bound1}
\end{align}
and 
\begin{align}
|S_w(y)-S_w(x)|&= \left|\sum_{j=1}^k w(y_j) - w(x_j)\right| \nonumber\\
& \leq \sum_{j=1}^k |w(y_j) - w(x_j)| \nonumber\\
&\leq M \sum_{j=1}^k |y_j-x_j| = M\|y-x\|_1 \nonumber\\
&\leq M\sqrt{k}\|x-y\| \label {Ineq:Bound2}
\end{align}
The last inequality follows by Cauchy-Schwartz inequality. Finally from equations \eqref{Ineq:Bound1}, \eqref{Ineq:Bound2} and \eqref{Lip-Bound}, we get
\begin{equation}\label{Bound-Contraction}
\|F(x)-F(y)\| \leq \frac{M\left(1+\sqrt{k}\right)}{S_w(y)}\|x-y\|
\end{equation}
{\bf Case (i):}
If $w(1)>0$, then we can write $S_w(y)\geq kw(1)$ and therefore from equation \eqref{Bound-Contraction} we get
\begin{equation}
\|F(x)-F(y)\| \leq \frac{M\left(1+\sqrt{k}\right)}{kw(1)}\|x-y\| \leq \frac{2M}{\sqrt{k}w(1)}\|x-y\|
\end{equation}
Thus $F$ is a contraction if $\sqrt{k}>\dfrac{2M}{ w(1)}$. \\
{\bf Case (ii):}
Assuming that $w$ is a convex function then for $\Delta_k = \{ y\in \mathbb{R}^{+k}: \sum_{i=1}^k y_i =1\}$
\[S_w(y) \geq kw(1/k), \; \forall y\in \Delta_k. \]
Therefore from equation \eqref{Bound-Contraction} we get 
\begin{equation}\label{Bound-Contraction-Convex}
\|F(x)-F(y)\| \leq \frac{M\left(1+\sqrt{k}\right)}{kw(1/k)}\|x-y\| \leq \frac{2M}{\sqrt{k}w(1/k)}\|x-y\|
\end{equation} 
Thus $F$ is a contraction if $\sqrt{k}w(1/k)>2M$. 
\end{proof}

The following proposition gives the sufficient condition for the stability of the equilibrium point $\dfrac{1}{k}\bone$, in case of a 
doubly stochastic replacement matrix.

\begin{proposition}\label{Prop-Stability}
Suppose $w$ is a non-increasing function on $[0,1]$ and $R$ is doubly stochastic matrix with eigenvalues $1, \lambda_1,\cdots,\lambda_s$ 
then, $\dfrac{1}{k} \bone$ is a stable equilibrium if 
\begin{equation}
\Re(\lambda_i)> \dfrac{kw\left(\dfrac{1}{k}\right)}{w'\left(\dfrac{1}{k}\right)}, \;\;\forall i =1,2,\cdots, s. 
\label{stability-condition-1}
\end{equation}
\label{Prop:Nonlinear-Stability-sufficient-condition}
\end{proposition}

\begin{proof}
Since the eigenvalues of the Jacobian matrix are $-1$ and $ b\lambda_i-1$ for $i = 1,2,\cdots, s$.
Thus equilibrium point $\dfrac{1}{k} \bone$ is stable, if only if
\begin{equation}
\Re(b\lambda_i-1)<0, \;\; \forall\; i=1, 2,\cdots, s. 
\end{equation}
$\iff $
\begin{equation}\label{stability-condition}
\Re(\lambda_s)> \frac{1}{b} = \frac{kw\left(\dfrac{1}{k}\right)}{w'\left(\dfrac{1}{k}\right)}.
\end{equation} 
This completes the proof.
\end{proof}

\begin{corollary}\label{Cor:Suffient-condition}
Notice that $\Re(\lambda_s) \geq -1$. Thus another sufficient condition for the stability is 
\begin{equation}
k>-\dfrac{w'\left(\dfrac{1}{k}\right)}{w\left(\dfrac{1}{k}\right)}.
\label{stability-condition-2}
\end{equation}
\end{corollary}

\begin{remark}
Assuming that $w(0),w'(0+)<\infty$, equation \eqref{stability-condition-1} or \eqref{stability-condition-2} hold for $k$ sufficiently large.
\end{remark}
\noindent

\section{Proofs} \label{Section-Proofs}

\begin{proof}[Proof of Theorem \ref{Thm:As-Convergence}]
Suppose $F$ is a contraction, then there exists a unique fixed point $y^*$ of $F$, such that $F(y^*)=y^*$.
Then 
\[ h(y^*) = 0.\]
that is, $y^*$ is also a unique equilibrium. Now using Theorem $2.$ and Corollary $3.$ from \cite{Borkar2008} (page 126) we get 
\[Y_n \to y^*,\;\;\;\text{as }n \to \infty.\] 
Now if $ \tilde{y}^*$ is an equilibrium point of the ODE in equation \eqref{ODE-Nn} that is $\tilde{y}^*$ satisfies 
\[ \tilde{y}^* = \frac{\bw(\tilde{y}^*R)}{S_w( \tilde{y}^*R)} \]
then
\[\tilde{Y}_n \to \tilde{y}^*,\;\;\;\text{as }n \to \infty.\] 
\end{proof}

\begin{proof}[Proof of Theorem \ref{Thm:DS-SLLN}]
We consider the linearized version of the non-linear ODE associated with $Y_n$, that is 
$\dot{y} = h\left(y\right)$, as in equation \eqref{ODE-Yn} around its equilibrium point 
$\dfrac{1}{k} \bone$, that is 
\begin{equation}
\dot{x} = \frac{\partial h }{\partial x} \left(\frac{1}{k}\bone \right) = Hx
\label{Linearized-ODE}
\end{equation}
where $H = bR-\dfrac{b}{k} J -I$. Assuming that $\lambda \neq \frac{1}{b}$ for any eigenvalue $\lambda$ of $R$, then 
$\left(I-bR\right)$ is invertible and then $x^* = \dfrac{-b}{k}\bone \left(I-bR\right)^{-1}$ is the unique equilibrium 
point of the linearized differential equation \eqref{Linearized-ODE}. By Hartman-Grobman Theorem (see \cite{Hartman1964} Chapter 9 and \cite{Grobman1962})
there exists a homeomorphism $f$ from a neighbourhood $U$ of $\frac{1}{k}\bone$ to a neighbourhood $V$ 
of $x^*$, such that $x(0) = f\left(y(0)\right)$ implies 
$x(t) = f\left(y(t)\right) \,\, \forall t>0$, where $x(t)$ is the solution of the linearized
ODE, and $y(t)$ is the solution of the non-linear ODE .
In particular, if the real part of all the eigenvalues of $H$ are
positive then $\frac{1}{k}\bone$ is asymptotically stable, that is $x(t)\to x^*$ almost surely and thus assuming the stability of $\dfrac{1}{k}\bone$ 
that is, if equation \eqref{stability-condition-1} holds then we get 
\[y(t) \to \frac{1}{k} \bone\,\,a.s.\]
This completes the proof for $Y_n$. This also proves  the convergence for $\tilde{Y}_n$, since the Jacobian matrices
for $h$ and $\tilde{h}$ at $\frac{1}{k} \bone$ are same.
\end{proof}

\begin{proof}[Proof of Theorem \ref{Thm:CLT-rho>1/2}]
Suppose $\rho>1/2$ then (see appendix Theorem \ref{Appn:CLT-rho}), 
\[\sqrt{n}\left(Y_n - \dfrac{1}{k} \bone\right) \implies N_k\left(0,\Sigma_1\right)\]
where 
\[\Sigma_1 = \int_{0}^{\infty} \left(e^{Hu} \right)^T \mathbf \Gamma_1 \left(e^{ Hu} \right) du,\]
with 
\[H = \left(bI - \dfrac{b}{k} J\right)R -\frac{1}{2} I.\]
and \begin{align*}
\Gamma_1 & = \lim_{n \to \infty} R^TE\left[ M_{n+1}^T M_{n+1}\Big\vert\mathcal{F}_n \right]R\\
& = \lim_{n \to \infty} R^TE\left[\left(\chi_{n+1}- \frac{\bw(Y_n)}{S_w(Y_n)} \right)^T \left(\chi_{n+1}- \frac{\bw(Y_n)}{S_w(Y_n)} \right) \Big\vert\mathcal{F}_n \right]R\\
& = \lim_{n \to \infty} R^T\left[E\left[\chi_{n+1}^T\chi_{n+1}\Big\vert\mathcal{F}_n \right] - \frac{\bw(Y_n)^T \bw(Y_n)}{S_w(Y_n)^2} \right]R\\
& = R^T\left[\frac{1}{k}I- \frac{1}{k^2} J\right]R.
\end{align*}
Now observe that $JR = RJ = J$, because $R$ is doubly stochastic. Therefore 
\begin{align*}
e^{uH} = e^ { bu R - \frac{bu}{k} JR-\frac{u}{2} I} = e^{bu R - \frac{bu}{k} J-\frac{u}{2} I}
\end{align*}
Again since $R$ commutes with $J$ and $I$, we can write 
\begin{align}
e^{uH} & = e^{buR} e^{- (bu/k) J} e^{-(u/2) I} \nonumber \\
& = e^{-u/2}\left[ \sum_{j=0}^\infty\frac{\left(\dfrac{-bu}{k} J\right)^j}{j!} \right] e^{ bu R } \nonumber \\
&= e^{-u/2}\left[ I + \frac{e^{-bu}-1}{k} J \right] e^{ bu R }
\end{align}
Now 
\begin{align}
e^{uH^T} \Gamma_1& = e^{-u/2}e^{ bu R^T }\left[ I + \frac{e^{-bu}-1}{k} J \right] R^T\left[\dfrac{1}{k}I- \frac{1}{k^2} J\right]R \nonumber \\
& = e^{-u/2}e^{ bu R^T }\left[ R^T + \frac{e^{-bu}-1}{k} J \right] \left[\dfrac{1}{k}R- \frac{1}{k^2} J\right] \nonumber \\
& = e^{-u/2}e^{ bu R^T }\left[ \dfrac{1}{k}R^TR - \frac{1}{k^2} J\right] 
  \end{align}
  
\begin{align}
e^{uH^T} \Gamma_1 e^{uH} & = e^{-u}e^{ bu R^T }\left[ \dfrac{1}{k}R^TR - \frac{1}{k^2} J\right ] \left[ I + \frac{e^{-bu}-1}{k} J \right]e^{ bu R } \nonumber \\
& = e^{-u}e^{ bu R^T }\left[ \dfrac{1}{k}R^TR - \frac{1}{k^2} J \right]e^{ bu R } \nonumber \\
& = e^{-u}\left[ \dfrac{1}{k}e^{ bu R^T } R^TR e^{ bu R }- \frac{e^{2bu}}{k^2} J \right] \nonumber \\
& = e^{-u}R^T\left[\frac{1}{k} e^{ bu R^T } e^{ bu R }- \frac{e^{2bu}}{k^2} J \right]R \label{Integral-expression}
\end{align}
The last step follows as $R$ and $e^{buR}$ commute.
Now we can rewrite the last expression as
\begin{align}
&= \dfrac{1}{k}R^T\left[ e^{-u} e^{ bu R^T } e^{ bu R }- \frac{e^{-u(1-2b)}}{k} J\right]R \nonumber \\
&= \dfrac{1}{k}R^T\left[ e^{-\frac{u}{2}( I -2b R^T)}e^{\frac{u}{2}(2bR-I)} - \frac{e^{-u(1-2b)}}{k}J \right]R \label{Equ:Var-Integrand}
\end{align}
Thus,
\begin{equation}
\int_0^\infty e^{uH^T} \Gamma_1 e^{uH} du =\frac{1}{k}R^T\left[ \Lambda_1- \frac{1}{k(1-2b)}J\right] R
\end{equation}

\noindent
where 
\begin{equation}
\Lambda_1= \int_0^\infty \exp \left(-u(1/2 I -b R^T) \right ) \exp \left(-u(bR-1/2I) \right ) du\\
\end{equation}
which satisfies the Sylvesters equation :
\begin{equation}\label{Equ:Sylv}
A \Lambda_1- \Lambda_1 B = I. 
\end{equation}
for $A = \frac{1}{2} I -b R^T$ and $B = bR - \frac{1}{2}I= -A^T$.
In particular, if $R$ is a normal matrix then 
\[\Lambda_1 = (A-B)^{-1} = (I-b(R+R^T))^{-1} \]
satisfies the Sylvesters equation, if and only if 
\begin{alignat*}{2}
&I=A(A-B)^{-1} - (A-B)^{-1}B \\
\iff &A-B = (A-B)A(A-B)^{-1} - B \\
\iff& A = (A-B)A(A-B)^{-1} \\
\iff &A(A-B) =(A-B)A \\
\iff &AB = BA \\
\iff &AA^T= A^TA \\
\iff & R^TR=RR^T.
\end{alignat*}
Therefore for a normal matrix $R$ 
\[\Sigma_1 = \frac{1}{k}R^T\left[ (I-b(R+R^T))^{-1} - \frac{1}{k(1-2b)}J\right] R.\]
Similarly, if $\rho >1/2$  we get
\[\sqrt{n}\left(\tilde{Y}_n - \frac{1}{k}\bone\right) \implies N(0,\tilde{\Sigma}_1)\]
where 
\[\tilde{\Sigma}_1 = \int_0^\infty \left(e^{uH} \right)^T \tilde \Gamma_1 \left(e^{ uH} \right) du,\]
for
\[H = \frac{\partial \tilde h}{\partial y}\Big\vert_{y=\frac{1}{k}\bone} +\frac{1}{2} I =bR - \dfrac{b}{k} J -\frac{1}{2} I \]
and \begin{align*}
\tilde \Gamma_1 & = \lim_{n \to \infty} E\left[ M_{n+1}^T M_{n+1}\Big\vert\mathcal{F}_n \right] = \left[\frac{1}{k}I- \frac{1}{k^2} J\right].
\end{align*}
Now similar to the expression obtained in equation $\eqref{Integral-expression}$ we get 
\begin{align}
e^{uH^T} \tilde \Gamma_1 e^{uH}& = e^{-u}\left[\frac{1}{k} e^{ bu R^T } e^{ bu R }- \frac{e^{2bu}}{k^2} J \right]
\end{align}
and therefore,

\begin{equation}
\tilde{\Sigma}_1 = \int_0^\infty e^{uH^T} \tilde \Gamma_1 e^{uH} du =\frac{1}{k}\left[ \Lambda_1- \frac{1}{k(1-2b)}J\right].
\end{equation}
where $\Lambda_1$ satisfies the Sylvesters equation \eqref{Equ:Sylv}.
\end{proof}

\begin{proof}[Proof of Theorem \ref{Thm:CLT-rho=1/2}]
Suppose $\rho =1/2$, then the result holds under the following two assumptions 
(see appendix Theorem \ref{Appn:CLT-rho})

\begin{enumerate}
\item \begin{equation}\label{Equ:Lindberg-conditon}
\frac{1}{n} \sum_{m=1}^n E\left[ \|M_m R\|^2 I\{ \|M_mR\|\geq \epsilon \sqrt{n}\} \Big\vert \mathcal{F}_{m-1}\right] \to 0. 
\end{equation}
a.s. or in $L^1$, for all $ \epsilon >0$.
\item For some $\epsilon >0$, as $y \to \dfrac{1}{k} \bone$
\begin{equation}\label{Equ:Taylor-Approx}
h(y) = h\left(\dfrac{1}{k}\bone \right)+ \left(y-\dfrac{1}{k}\bone\right) Dh\left( \dfrac{1}{k} \bone \right) + o\left(\left\|y- \dfrac{1}{k} \bone \right\|^{1+\epsilon}\right) 
\end{equation}
\end{enumerate}
The Linderberg condition in equation \eqref{Equ:Lindberg-conditon} holds as $\| M_mR\| \leq k(k+1) $ for all $m$, 
and for $\sqrt{n}>\frac{k(k+1)}{\epsilon}$, $I\{ \|M_mR\|\geq \epsilon \sqrt{n}\} = 0$ for all $m$.
The second condition \eqref{Equ:Taylor-Approx} is also satisfied for a twice differentiable function $h$. 
Thus for $\rho = 1/2$, (see appendix Theorem \ref{Appn:CLT-rho}), we have
\[ \frac{\sqrt{n}}{\log n^{\nu-1/2}}\left( Y_n - \dfrac{1}{k}\bone \right)\implies N_k\left(0,\Sigma_2\right)\]
where
\[\Sigma_2 = \lim_{n \to \infty} \frac{1}{(\log n )^{2\nu-1}}\int_0^{\log n } e^{uH^T} \mathbf \Gamma_1 e^{uH} du.\]
Now from equation \eqref{Equ:Var-Integrand}, we get

\begin{align*}
\Sigma_2 &= \frac{1}{k}R^T\left[ \Lambda_2- \lim_{n\to \infty }\frac{1-n^{-(1-2b)}}{(\log n )^{2\nu-1}}J\right] R\\
&= \frac{1}{k}R^T \Lambda_2R 
\end{align*}
where 
\begin{equation} \label{Integral-2}
\Lambda_2=\lim_{n \to \infty} \frac{1}{(\log n )^{2\nu-1}}\int_0^{\log n }\exp \left(-u(1/2 I -b R^T) \right ) \exp \left(-u(bR-1/2I) \right ) du.
\end{equation}
Similarly for $\tilde{Y}_n$ the required Linderberg condition holds that is 
\begin{equation}
\frac{1}{n} \sum_{m=1}^n E\left[ \| M_m \|^2 I\{ \| M_m\|\geq \epsilon \sqrt{n}\} \Big\vert \mathcal{F}_{m-1}\right] \to 0. 
\end{equation}
and therefore by Theorem \ref{Appn:CLT-rho} in the appendix we get
\begin{align*}
\tilde \Sigma_2 &= \frac{1}{k}\left[ \Lambda_2- \lim_{n\to \infty }\frac{1-n^{-(1-2b)}}{(\log n )^{2\nu-1}}J\right] = \frac{1}{k}\Lambda_2 
\end{align*}
where $\Lambda_2$ is as given in equation \eqref{Integral-2}.

\end{proof}

\begin{proof}[Proof of Theorem \ref{Thm:rho<1/2}]
Proof of this theorem follows from Theorem \ref{Appn:CLT-rho} in the appendix.
\end{proof}

\section{Examples}\label{Section-Examples}

\subsection{Linear weight function}\mbox{}\\
Let $w:[0,1] \to \mathbb{R}^+$ be such that:
\[w(y)= \theta-y;\;\;\text{ for }\theta \geq 1,\; \text{ and }y \in [0,1].\]
Then the stochastic approximation algorithm in equation \eqref{SA-recursion} holds with
\begin{equation}\label{Equ:Linear-h}
h(y) = y\left(AR-I\right) 
\end{equation}
for \[A_{k\times k} = \dfrac{1}{k\theta-1}\left(\theta J-I\right).\]
Notice that, an equilibrium point of the associated ODE is also a stationary distribution of $AR$. Thus 
if $AR$ is irreducible and $y^*$ is its unique stationary distribution then by Theorem \ref{Thm:As-Convergence}, we get
\[ Y_n \longrightarrow y^*\;\; a.s..\] 
The above almost sure convergence result was also proved in Theorem $1$ of \cite{BKaur17-1}. In fact, in \cite{BKaur17-1} 
a necessary and sufficient condition for $AR$ to be irreducible is given and also convergence for the case when $ AR $  is reducible is obtained. \\

For the central limit theorem, consider a doubly stochastic matrix $R$. 
In this case, the constant $b$, as defined in equation \eqref{Def:constant-b} is
\[b = -\frac{1}{k\theta-1}\]
and $\rho$ as defined in equation \eqref{Def:rho} is 
\[\rho = 1+\dfrac{\Re(\lambda_s) }{k\theta-1}. \]
Now we separately consider $k=2$ and $k\geq 3$, in order to identify the possible values of $\rho$.  
First let  $k=2,$ and $ R = \begin{bmatrix}
p&1-p\\
1-p&p
\end{bmatrix}$, where $p \in [0,1]$. Then $R$ has eigenvalues $1$ and $2p-1$ and 
\[\rho = 1+\frac{2p-1}{2\theta-1}.\]
Thus,
\begin{equation}
\rho \geq \frac{1}{2}\;\; \iff \;\;2p -1 \geq \frac{1-2\theta}{2}
\label{Equ:k=2-linear}
\end{equation}
Therefore, by Theorem \ref{Thm:CLT-rho>1/2} and Theorem \ref{Thm:CLT-rho=1/2}, we get
\begin{equation}
\sigma_{n}\left(Y_{n,1} - \dfrac{1}{2}\right) \implies N\left(0,\sigma^2\right)
\end{equation}
where $\sigma_n = \begin{cases}
\sqrt{\dfrac{n}{\log n}} &\text{ if }\text{the eigenvalue of } R \text{ (other than $1$) is equal to } \frac{1-2\theta}{2}, \\
\sqrt{n} &\text{ if }\text{the eigenvalue of } R \text{ (other than $1$) is }> \frac{1-2\theta}{2}. \\
\end{cases}$\\

As shown in the figure below, $\rho$ is equal to $\frac{1}{2}$ on the highlighted straight line, and below this line $\rho$ is less than $\frac{1}{2}$. 
Note that for a large 
range of minimum eigenvalue of $R$ and parameter $\theta$, $\rho$ is greater than $\frac{1}{2}$ for which the asymptotic
normality holds with scaling factor $\sqrt{n}$.

\begin{center}
\includegraphics[scale=.6]{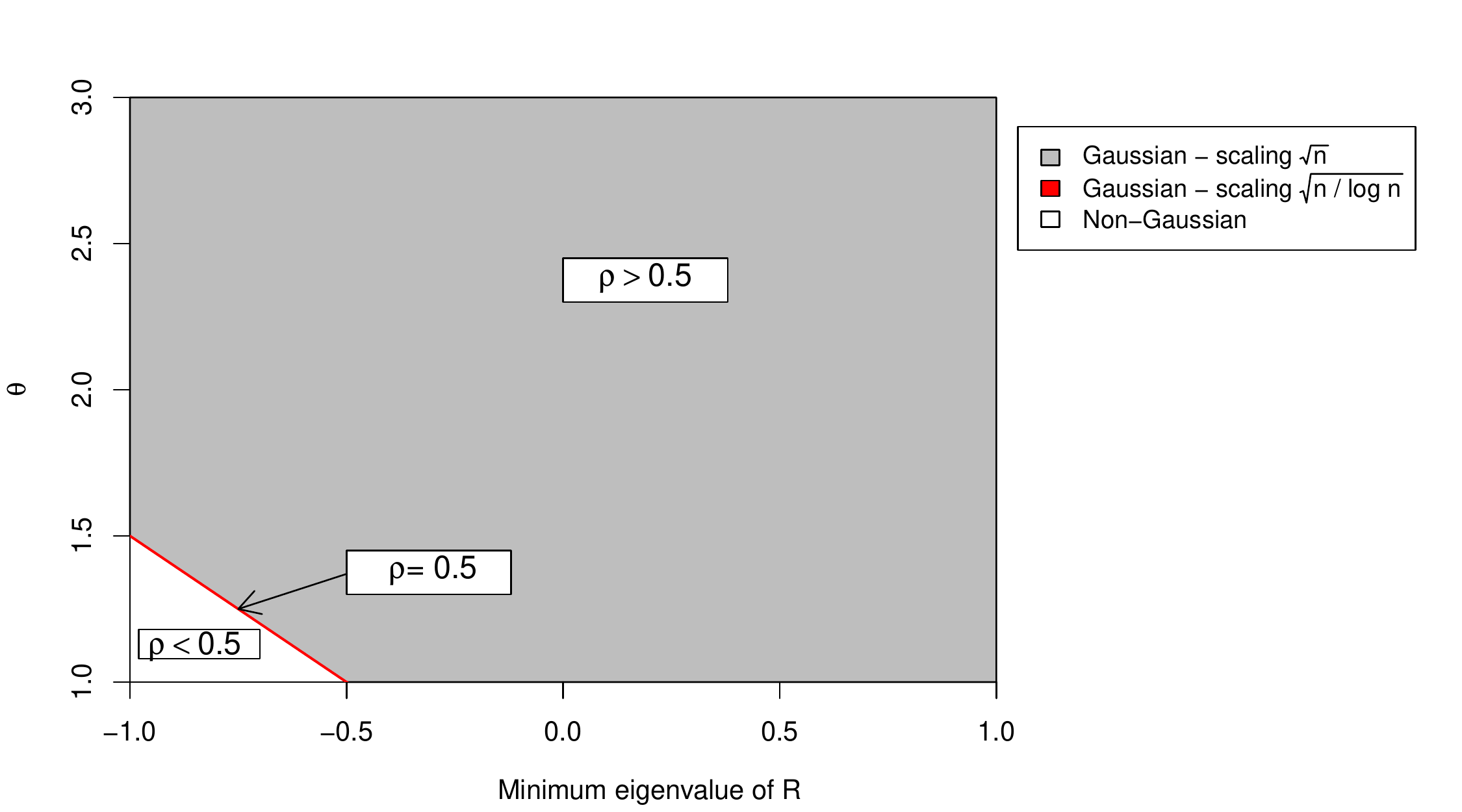}
\captionof{figure}{Range of $\rho$ for given $\theta$ and minimum eigenvalue of $R$}
\end{center}

Now for $k\geq 3$, using the fact that $\Re(\lambda_s) \geq -1$, we get
\begin{align*}
\rho &\geq 1-\dfrac{1}{k\theta-1}\\
\end{align*}
and therefore for $k\geq 3 $, $\rho \geq 1/2$. Thus, there is no non-Gaussian limiting behavior. 
In fact, in this case we always have Gaussian limit with $\sqrt{n}$ scaling, except when $\rho = \dfrac{1}{2}$, which can only happen when $k=3$ and 
then
\[\rho = 1/2\;\;\;\iff \;\;\;\Re(\lambda_s) = -\frac{3\theta-1}{2}.\]
Which is possible only if $\theta =1$ and $\Re(\lambda_s) = -1$, 
and for a $3\times 3 $ stochastic matrix, there can only be at most one eigenvalue with real part equal to $-1$.

The above result for $k\geq 2$ and $\rho \geq \dfrac{1}{2}$ has also been obtained in Theorem $2$ of \cite{BKaur17-1}.
In fact, central limit theorem for a general class of replacement matrices is given in \cite{BKaur17-1}.

\subsection{Inverse power law weight function} \mbox{}\\ 
Let 
\begin{equation} \label{Equ:exp-w}
w(x) = (\theta+ x)^{-\alpha}, \text{ for }\theta,\alpha>0 
\end{equation} 
and $R$ be a doubly stochastic matrix. 
Then $b = -\dfrac{\alpha}{k\theta +1} $ and therefore by Proposition \ref{Prop-Stability},
$\dfrac{1}{k}\bone$ is a stable equilibrium point if 
\begin{equation}
\Re(\lambda_i) > - \frac{k\theta+1}{\alpha} \;\;\;\text{ for } i = 1,2,\cdots ,s.
\label{Ex:3-Stability}
\end{equation}
In particular, the above condition for stability holds if $R = I$ or if $\alpha < k\theta+1$. Also,
\[ \rho = 1+\dfrac{\alpha}{k\theta+1}\Re(\lambda_s).\]
and then 
\[ \rho \geq 1/2 \iff \Re(\lambda_s)\geq -\frac{k\theta+1}{2\alpha}.\]
Therefore, the scaling for central limit theorems depend on the values of $\alpha$ and $\theta$. \\
In particular, for $\alpha =1$ the condition for stability in equation \eqref{Ex:3-Stability} holds
for $k\geq 2$ and thus by Theorem \ref{Thm:DS-SLLN} we get
\[Y_n \longrightarrow \frac{1}{k} \bone\;\; a.s. \] 
%
%
%
For the central limit theorem, as shown in the figure below $\rho$ takes value more than $\frac{1}{2}$ in the shaded region and thus 
for given $\theta$ and $R$, the central limit theorems hold accordingly.

\includegraphics[scale=0.7]{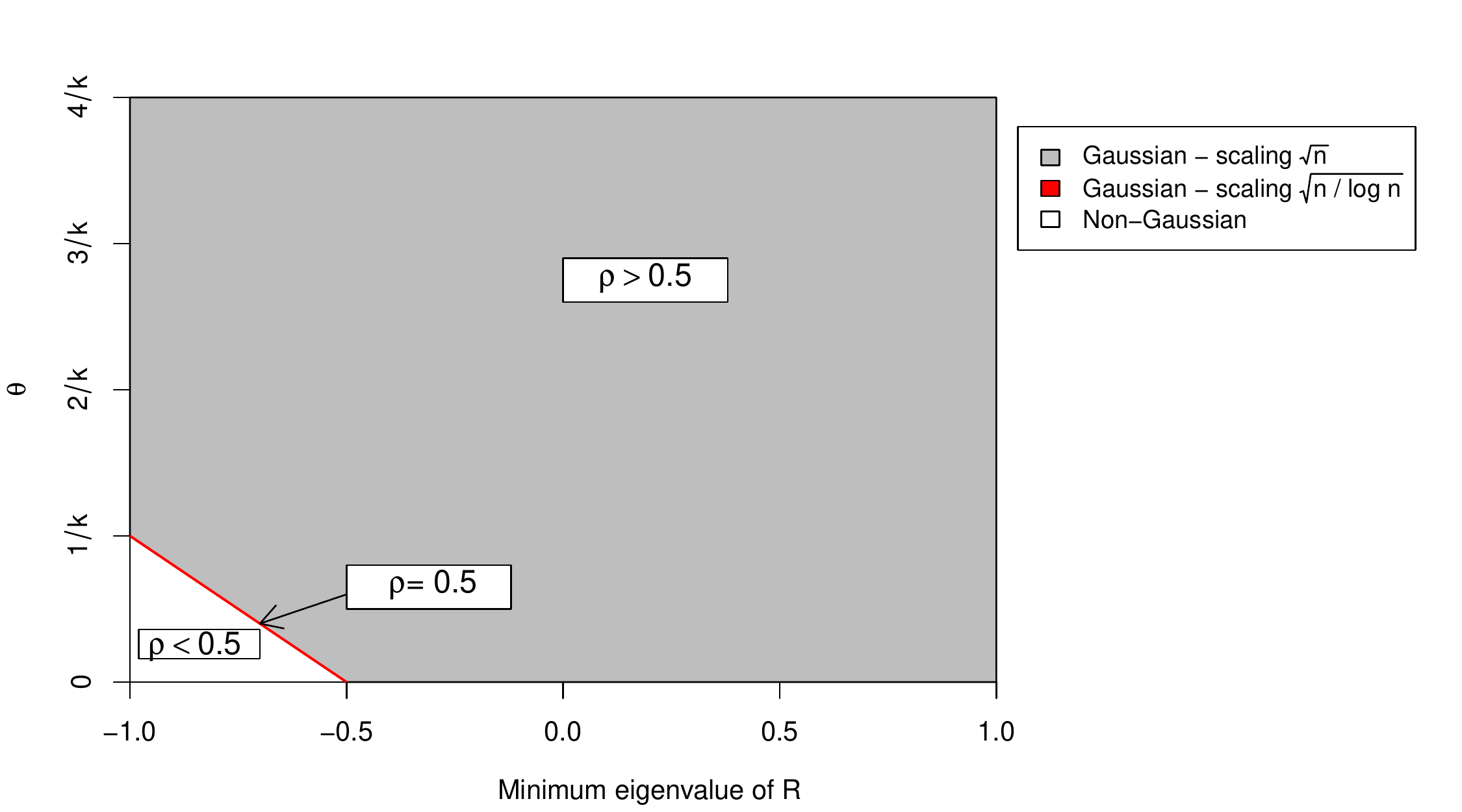} 
\captionof{figure}{Range of $\rho$ for given $\theta$ and minimum eigenvalue of $R$}

Note that, above a critical value for $\theta$, we observe asymptotic
normality with the scaling factor of $\sqrt{n}$ for any choice of replacement matrix $R$.
Further, the region for $\rho<\dfrac{1}{2}$ decreases as we increase the number of colours $k$. \\

Now we show that, the uniform vector $\dfrac{1}{k} \bone$ is not necessarily a stable equilibrium for any doubly stochastic matrix. Let
\[R = \begin{bmatrix} 0&0&0&1\\0&0&1&0\\0&1&0&0\\1&0&0&0\end{bmatrix},\] 
then the condition for stability in equation \eqref{Ex:3-Stability}
is not satisfied for $\alpha >k\theta +1$, and for this choice of $\alpha$ the equilibrium point $\dfrac{1}{4}\bone$ is an unstable point.
In fact, by Theorem $1$ of \cite{Pem90} one can show that in this case
\[ P\left(Y_n\to \frac{1}{4}\bone\right) = 0. \]

\subsection{Exponential weight function} \mbox{}\\
Let
\[w(x) = \exp\left(-\dfrac{x}{\theta}\right ), \text{ for }\theta>0\]
then 
\[b = - \frac{1}{k\theta}\;\;\;\text{ and } \rho = 1+\frac{ \Re(\lambda_s)}{k\theta}\]
and thus $\dfrac{1}{k} \bone $ is a stable equilibrium for a doubly stochastic matrix if 
\[ \Re(\lambda_i) > -k\theta; \;\;\text{ for } i=1,2,\cdots, s.\]
and 
\[\rho \geq 1/2 \iff \Re(\lambda_s) \geq -\frac{k\theta}{2}\] 
As shown in the following graph, $\rho > \dfrac{1}{2}$ in the shaded region and equal to $\dfrac{1}{2}$ on the straight line.

\begin{center}
\includegraphics[scale=0.7]{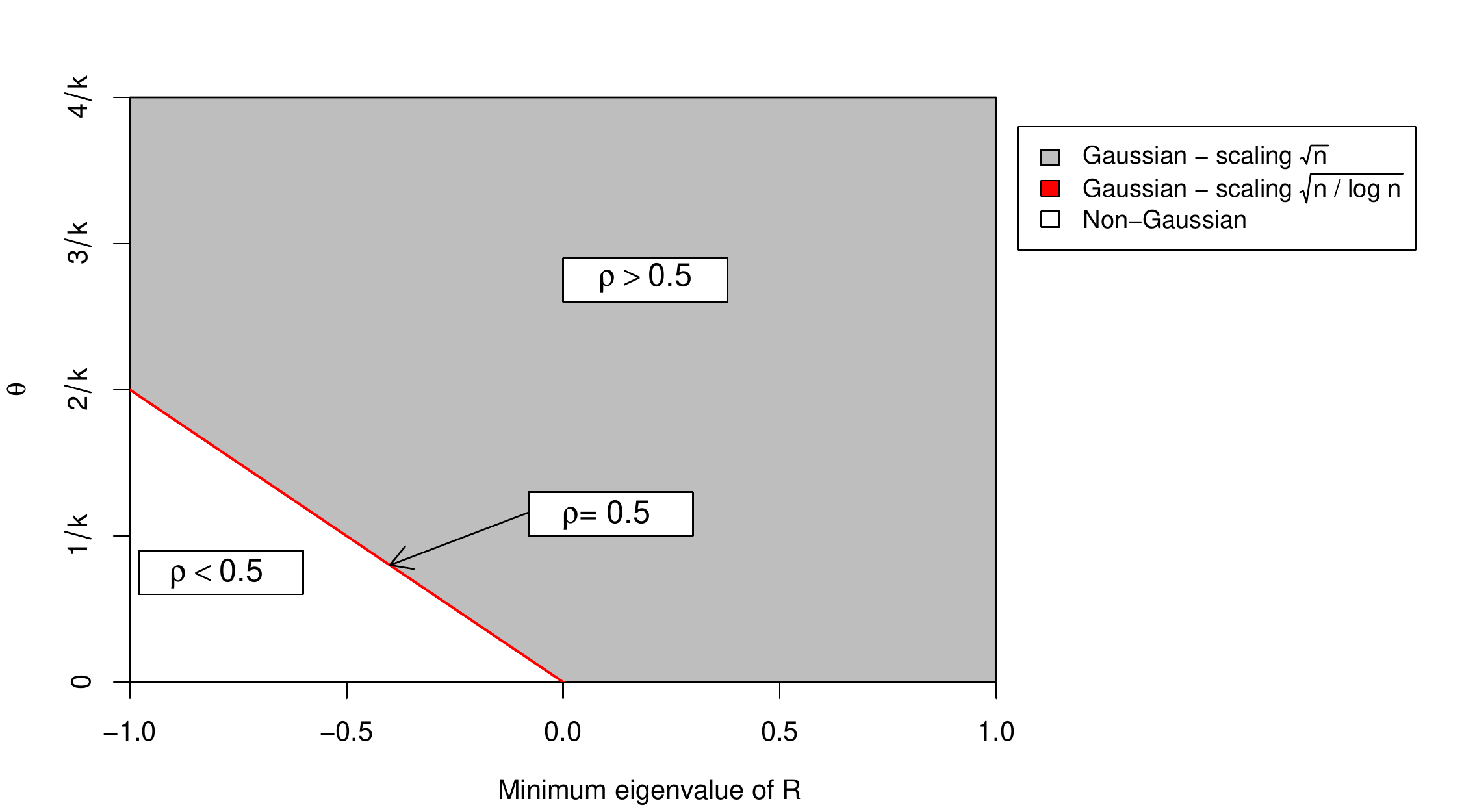} 
\captionof{figure}{Range of $\rho$ for given $\theta$ and minimum eigenvalue of $R$}
\end{center}

 \appendix
\section{Appendix}
In this section we state some of the general results in the stochastic approximation theory, 
for the discrete stochastic approximation algorithm $X_n$ in $\mathbb{R}^d$ (as defined in Section \ref{SAalgorithm}), satisfying  
\[ X_{n+1} = X_n +\gamma_{n+1} h\left(X_n\right)+\gamma_{n+1} M_{n+1}R.\]

\begin{definition}
A set $A$ is called \emph{stable} (or attractor) if for each $\epsilon >0$ there is a $\delta>0$ such that all trajectories 
starting in  $N_\delta (A)$ never leave $N_\epsilon(A)$.
\end{definition}
When $h$ is differentiable, an equivalent definition of a stable (or attractor)  equilibrium point is given below.
\begin{definition}(Stable/unstable equilibrium point). An equilibrium $x^*$ is called
stable (or attractor) if all the eigenvalues of $Jh(x^*)$ have negative
real part, and it is called unstable otherwise.
  \end{definition}

\begin{theorem}[Almost sure convergence] \label{Thm:AS-convergence}
 Assume that $h$ is a Lipschitz function and 
 \[ \sup_{n\geq n_0} \mathbb{E}\left[\|M_{n+1}\|^2\vert \mathcal{F}_n\right]< \infty \;\; a.s., \]
 and $\gamma_n$ is a positive sequence satisfying 
 \[ \sum_{n=1}^\infty \gamma_n = \infty, \;\;\;\text{ and } \;\;\;\sum_{n=1}^\infty \gamma_n^2<\infty \]
 On the event $A_\infty = \{ \omega\; \vert\; h(X_n(\omega))_{n \geq 0} \text{ is bounded } \}$, $\mathbb{P}(dw)- a.s.$  then 
\begin{itemize}
\item [(a)] (Theorem 5.7 \cite{Benaim99}) The limit set of $X(t)$  that is,
\[L(X(t)) = \cap_{t\geq 0}\overline{ X[t,\infty) } \]
is almost surely an internally chain
 transitive set for the unique solution $\phi(x_0 , t)$ of the mean limit ODE 
\begin{equation} \label{ODE-A}
 \dot{x}=  h(x).
\end{equation}

 \item[(b)] If the only internally chain transitive invariant set are isolated equilibrium points of $h$ then 
 $X_n$ converges a.s. to the set of equilibrium points of $h$.\\
 \item[(c)] If there is a unique equilibrium point that is $\{h = 0\} = \{x^*\}$ and $\phi(x_0 , t) \to x^*$ as $t \to \infty$ locally uniformly in $x_0$ , then 
 \[X_n \to x^* \;\;as\;\; n \to \infty\; a.s..\]
\item[(c)] If $k = 2$ and $\{h = 0\}$ is locally finite, then 
\[L(X(t))  \subset \{h = 0\} \]
i.e. 
\[X_n \to X_\infty \in \{h =0\}.\]
\end{itemize}
\end{theorem}
\noindent
%
%

Let $\lambda_1, \lambda_2, \cdots, \lambda_p$ be the eigenvalues of $Dh(x^*)$
with $dim(J_t) = \nu_t $ for every $t =1,2,\cdots, p$.
Let \[ \rho = \min_{1\leq i\leq p} \{ -Re(\lambda_i )\} \;\;\;\text{ and }\;\;\;\;\nu = \max \{\nu_t:Re(\lambda_t) = \rho \}.\]

We will need the following assumptions to state the central limit theorems for different values of $\rho$.
\begin{enumerate}
  \item[A1 ] $x^*$ is a stable equilibrium.
   \item[A2 ] The Lindeberg condition 
   \[\frac{1}{n}\sum_{m=1}^n E\left[ \| M_m\|^2 \mathbf{I} \{ \|M_m\|\geq  \epsilon \sqrt{n}\} \Big\vert \mathcal{F}_{m-1}  \right]\to 0 \;\; a.s. \]   
   \item[A3 ] $\dfrac{1}{n}\sum_{m=1}^n E\left[M_{m+1}^T M_{m+1}\Big\vert \mathcal{F}_m \right] \to\mathbf \Gamma$ \;\;\; a.s. or in $L_1$
  \end{enumerate}

\begin{theorem}[Scaling Limits \cite{Zhang2016}] \label{Appn:CLT-rho}
  Suppose   $X_n \to x^*$ a.s. then 
  \begin{enumerate}
  \item if $\rho >1/2$, and the following two conditions are satisfied 
  where $\mathbf \Gamma$ is deterministic symmetric positive semidefinite matrix, and assumptions $A1$, $A2$, and $A3$ hold then
  \[\sqrt{n}\left( X_n  - x^* \right) \cd N\left(0, \Sigma_1\right) \]
  where 
  \begin{equation}\label{limit-Var-Cov-matrix}
   \Sigma_1 = \int_0^\infty \left(e^{ - \left(Dh(x^*)+ 1/2I\right)u}\right)^t\mathbf \Gamma e^{ - \left(Dh(x^*)+1/2I\right)u} du  
  \end{equation}
   \item  Suppose $\rho =1/2$ and $w$ is a twice differentiable function and  assumption $A2$  holds then
   
  \[\frac{\sqrt{n}}{\log n^{\nu -1/2}}\left( X_n -x^*\right ) \implies N\left( 0,\Sigma_2 \right). \]
  where 
  \[\Sigma_2  = \lim_{n\to\infty} \frac{1}{(\log n)^{2\nu-1}} \int_0^{\log n} \left(e^{ - \left(Dh(x^*)- 1/2I\right)u}\right)^t\mathbf \Gamma e^{ - \left(Dh(x^*)- 1/2I\right)u} du \]
  \item Suppose $0<\rho <1/2$, $w$ is twice differentiable and assumption $A3$ holds then
  \[\frac{n^{\rho}}{\log n ^{\nu -1}} \left(X_n - x^*\right) - X_n \longrightarrow 0 \]
  where $X_n$ is random vector defined as
  \[X_n = \sum_{i: Re(\lambda_i) = \rho} e^{-iIm(\lambda_i)\log n }\xi_i v_i\]
   and $v_i$ is the left eigenvector of $Dh(x^*)$ with respect to the eigenvalue $\lambda_i$.
  
    \end{enumerate}
 \end{theorem}

%

\bibliographystyle{abbrv}
\bibliography{ref}

\def\cprime{$'$} \def\cprime{$'$} \def\cprime{$'$} \def\cprime{$'$}
  \def\cprime{$'$} \def\cprime{$'$} \def\cprime{$'$}
\begin{thebibliography}{10}

\bibitem{BaiHu05}
Z.-D. Bai and F.~Hu.
\newblock Asymptotics in randomized urn models.
\newblock {\em Ann. Appl. Probab.}, 15(1B):914--940, 2005.

\bibitem{BKaur17-1}
A.~Bandyopadhyay and G.~Kaur.
\newblock Generalized p\'olya urn schemes with negative but linear
  reinforcements, (preprint).
\newblock 2017.

\bibitem{BaTh14b}
A.~Bandyopadhyay and D.~Thacker.
\newblock Rate of convergence and large deviation for the infinite color
  {P}\'olya urn schemes.
\newblock {\em Statist. Probab. Lett.}, 92:232--240, 2014.

\bibitem{BaTh14a}
A.~Bandyopadhyay and D.~Thacker.
\newblock P\'{o}lya urn schemes with infinitely many colors.
\newblock {\em Bernoulli}, 23(4B):3243--3267, 2017.

\bibitem{Benaim99}
M.~Bena\"\i~m.
\newblock Dynamics of stochastic approximation algorithms.
\newblock In {\em S\'eminaire de {P}robabilit\'es, {XXXIII}}, volume 1709 of
  {\em Lecture Notes in Math.}, pages 1--68. Springer, Berlin, 1999.

\bibitem{RBhatia}
R.~Bhatia.
\newblock {\em Matrix analysis}, volume 169 of {\em Graduate Texts in
  Mathematics}.
\newblock Springer-Verlag, New York, 1997.

\bibitem{Borkar2008}
V.~Borkar.
\newblock {\em Stochastic Approximation A Dynamical Systems Viewpoint}.
\newblock 2008.

\bibitem{maulik1}
A.~Bose, A.~Dasgupta, and K.~Maulik.
\newblock Multicolor urn models with reducible replacement matrices.
\newblock {\em Bernoulli}, 15(1):279--295, 2009.

\bibitem{maulik2}
A.~Bose, A.~Dasgupta, and K.~Maulik.
\newblock Strong laws for balanced triangular urns.
\newblock {\em J. Appl. Probab.}, 46(2):571--584, 2009.

\bibitem{DasMau11}
A.~Dasgupta and K.~Maulik.
\newblock Strong laws for urn models with balanced replacement matrices.
\newblock {\em Electron. J. Probab.}, 16:no. 63, 1723--1749, 2011.

\bibitem{Grobman1962}
D.~M. Grobman.
\newblock Topological classification of neighborhoods of a singularity in
  {$n$}-space.
\newblock {\em Mat. Sb. (N.S.)}, 56 (98):77--94, 1962.

\bibitem{Hartman1964}
P.~Hartman.
\newblock {\em Ordinary differential equations}.
\newblock John Wiley \& Sons, Inc., New York-London-Sydney, 1964.

\bibitem{Svante1}
S.~Janson.
\newblock Functional limit theorems for multitype branching processes and
  generalized {P}\'olya urns.
\newblock {\em Stochastic Process. Appl.}, 110(2):177--245, 2004.

\bibitem{KushnerClark78}
H.~J. Kushner and D.~S. Clark.
\newblock {\em Stochastic approximation methods for constrained and
  unconstrained systems}, volume~26 of {\em Applied Mathematical Sciences}.
\newblock Springer-Verlag, New York-Berlin, 1978.

\bibitem{LaPa13}
S.~Laruelle and G.~Pag{\`e}s.
\newblock Randomized urn models revisited using stochastic approximation.
\newblock {\em Ann. Appl. Probab.}, 23(4):1409--1436, 2013.

\bibitem{LucMcDiar2005}
M.~J. Luczak and C.~McDiarmid.
\newblock On the power of two choices: balls and bins in continuous time.
\newblock {\em Ann. Appl. Probab.}, 15(3):1733--1764, 2005.

\bibitem{LucMcDiar2006}
M.~J. Luczak and C.~McDiarmid.
\newblock On the maximum queue length in the supermarket model.
\newblock {\em Ann. Probab.}, 34(2):493--527, 2006.

\bibitem{Pem90}
R.~Pemantle.
\newblock Nonconvergence to unstable points in urn models and stochastic
  approximations.
\newblock {\em Ann. Probab.}, 18(2):698--712, 1990.

\bibitem{Pe90}
R.~Pemantle.
\newblock A time-dependent version of {P}\'olya's urn.
\newblock {\em J. Theoret. Probab.}, 3(4):627--637, 1990.

\bibitem{Polya30}
G.~P{\'o}lya.
\newblock Sur quelques points de la th\'eorie des probabilit\'es.
\newblock {\em Ann. Inst. H. Poincar\'e}, 1(2):117--161, 1930.

\bibitem{Zhang2016}
L.-X. Zhang.
\newblock Central limit theorems of a recursive stochastic algorithm with
  applications to adaptive designs.
\newblock {\em Ann. Appl. Probab.}, 26(6):3630--3658, 2016.

\end{thebibliography}

\end{document}